\documentclass[12pt]{article}
\usepackage[margin=1in]{geometry}
\usepackage{amsmath}
\usepackage{amssymb}
\usepackage{amsthm}
\usepackage{enumerate}
\usepackage{color,xcolor,soul}

\newtheorem{theorem}{Theorem}[section]

\newtheorem{corollary}[theorem]{Corollary}
\newtheorem{observation}[theorem]{Observation}
\newtheorem{example}[theorem]{Example}
\newtheorem{definition}[theorem]{Definition}
\newtheorem{lemma}[theorem]{Lemma}

\newtheorem{remark}[theorem]{Remark}
\newtheorem{proposition}[theorem]{Proposition}

\title{Gap-Sums via Quasi-Arithmetic Means with Applications to Fibonacci and Lucas Sequences}
\author{
 Omid Khormali$^{1,}$\footnote{Corresponding author: ok16@evansville.edu} \quad Ghaya Mtimet$^{1}$ \quad Nuh Aydin$^{2}$ \\Mohammad K. Azarian$^{1}$ \\[1em]
{\small $^{1}$Department of Mathematics, University of Evansville, Evansville, IN 47722, USA}\\
{\small Emails: ok16@evansville.edu, gm163@evansville.edu, ma3@evansville.edu,}\\
{\small $^{2}$Department of Mathematics and Statistics, Kenyon College, Gambier, OH 43022, USA}\\
{\small Email: aydinn@kenyon.edu}
}

\begin{document}
\maketitle
\begin{abstract}
We develop a unified framework for studying the integers missing between consecutive terms of an increasing integer sequence, extending Barry's arithmetic gap-sum to geometric and harmonic analogues via the theory of quasi-arithmetic means. All three gap-sums admit a common interpretation: each equals the gap size multiplied by the appropriate mean of the missing integers. Building on this, we prove a general sparse summation theorem expressing the sum of a strictly monotonic function over a sparse integer sequence as the full range sum corrected by the gap-sums of the missing portions. Specializing on the three Pythagorean means recovers a classical formula of al-K\=ash\={\i} from the fifteenth century in the arithmetic case, and yields explicit formulas in the geometric and harmonic cases. As a concrete application of the geometric case, we derive a product identity involving the Fuss--Catalan numbers. Applying the harmonic case to the Fibonacci and Lucas sequences, we establish that the harmonic gap-sum converges to $\ln(\alpha)$ exponentially, where $\alpha$ is the golden ratio, and derive explicit two-term asymptotic expansions for the tails of the reciprocal Fibonacci and Lucas series with closed-form coefficients, and establish the asymptotic formula $H_{u_n} \sim n\ln(\alpha)$ for both $u_n = F_n$ and $u_n = L_n$, with explicit $O(1)$ error terms that differ due to their distinct initial conditions. As a further consequence, by comparing the gap-sum expansions with the classical Hardy--Wright expansion of harmonic numbers, we derive exact series identities expressing Euler's constant $\gamma$ in terms of harmonic numbers at Fibonacci and Lucas indices, and obtain a new identity relating the reciprocal 
Fibonacci constant $\psi$ and the reciprocal Lucas constant $\psi_L$.
\end{abstract}



\section{Introduction}

The summation of arithmetic sequences is one of the oldest problems in mathematics, with roots traceable to ancient Babylonian, Greek, and Indian traditions. The formula for the sum of an arithmetic progression was known to Diophantus and Brahmagupta, among others, long before the medieval period. Nevertheless, it was the Persian mathematician and astronomer Ghiy\={a}th al-D\={i}n Jamsh\={i}d al-K\=ash\={\i} (c.\ 1380--1429), working at the celebrated observatory of Samarqand,~\footnote{Samarqand, 
also spelled Samarkand, is a city in present-day Uzbekistan.} under the patronage of Ulugh Beg, who gave one of the most systematic and pedagogically complete treatments of finite summation in his encyclopedic \emph{Mift\=ah al-His\=ab} (The Key to Arithmetic, 1427)~\cite{alkashi_miftah}.~\footnote{The title appears in two transliterations in the literature: \textit{Mift\=a\d{h} al-\d{H}is\=ab}, following standard Arabic 
romanization (as in \cite{alkashi_miftah}), and \textit{Meftah al-hesab}, used in \cite{Az} reflecting Persian phonological conventions. Both refer to the same work.} This work covered numerous topics, including algebra, arithmetic, astronomy, geometry, and trigonometry, and remained influential for more than 500 years~\cite{Az}. Within the Islamic scholarly tradition, the Mift\=ah al-His\=ab was widely regarded as a principal reference in mathematics and served as a foundational textbook across a broad range of disciplines \cite{alkashi_miftah}. This historical significance makes it all the more worthwhile to revisit al-K\=ash\={\i}'s contributions and explore what further mathematics they suggest. More on al-K\={a}sh\={\i}'s broader mathematical contributions and the content of \textit{Mift\={a}\d{h}} can be found in~\cite{nuh, Miftah1, Miftah2, alkashi_miftah,  Az1, Az2, Az3, Az4, Az6, Az}.\\


In Mift\=ah al-His\=ab \cite[p. 75]{alkashi_miftah}, al-K\={a}sh\={\i}'s seventh rule \footnote{The seventh rule from \cite[p. 90]{Az}: If $a$, $d$, $l$, and $S_n$ are the first term, the increment, the $n$th term, and the sum of the first $n$ terms of an arithmetic progression respectively, then $l = a + (n-1)d$, and $S_n = \frac{n(a+l)}{2}.$
}
computes the sum of an arithmetic sequence $(u_i)_{i\ge 1}$ with constant increment $r$, satisfying $u_{i+1}=u_i+r$, as
\[
S_n = \sum_{i=1}^{n} u_i = \frac{(u_1 + u_n)\,n}{2},
\]
which expresses the total as the arithmetic mean of the first and last terms, scaled by the number of terms. From this rule one recovers, in particular, the classical identity
\[
1+2+\cdots+n=\frac{n(n+1)}{2},
\]
and more generally, for an arithmetic progression with first term $a$ and common difference $d$,
\[
\sum_{i=0}^{n-1}(a+id) = \frac{n}{2}\left(2a+(n-1)d\right).
\]
The arithmetic mean is thus the natural backbone of 
al-K\=ash\={\i}'s summation framework. From a modern perspective, the arithmetic mean is one among three classical Pythagorean means, whose origins trace back to ancient Greek mathematics. The arithmetic, geometric, and harmonic means were known in the Pythagorean tradition, acknowledged by Pythagoras, Plato, and Aristotle. They were later systematically treated by Nicomachus 
of Gerasa (c.~60--120\,AD) in his \textit{Introductio rithmetica}~\cite{Nicomachus}. For two positive real numbers $x$ and $y$, these are defined as
\[
A(x,y)=\frac{x+y}{2},\qquad
G(x,y)=\sqrt{xy},\qquad
H(x,y)=\frac{2xy}{x+y},
\]
corresponding respectively to the arithmetic, geometric, and harmonic means. More generally, all three fit within the broader framework of quasi-arithmetic means \cite{NielsenInductiveMean, Bullen2003}: for a strictly monotonic function $f$, the quasi-arithmetic mean is defined by
\[
M_f(x,y) = f^{-1}\!\left(\frac{f(x)+f(y)}{2}\right).
\]

While Nielsen \cite{NielsenInductiveMean} restricts to strictly increasing functions, Bullen \cite[Ch. IV, Def. 1]{Bullen2003} allows strictly monotonic functions, which include both increasing and decreasing cases. The three Pythagorean gap-sums thus correspond to $f(x) = x, \log(x), 1/x$ respectively. \\

A complementary perspective on integer sequences comes from asking not about the terms that are present, but about those that are missing. For an increasing sequence of positive integers $(a_n)_{n\ge 0}$, consecutive terms $a_n$ and $a_{n+1}$ are generally not adjacent integers, leaving a gap of integers strictly between them. 
In \cite{BarryGapSum}, Barry introduced the \emph{gap-sum} and \emph{gap-product} sequences, defined respectively by
\[
GS_n=\sum_{j=1}^{a_{n+1}-a_n-1}(a_n+j)
\quad\text{and}\quad
GP_n=\prod_{j=1}^{a_{n+1}-a_n-1}(a_n+j),
\]
which capture the local arithmetic structure of each gap individually. Notably, Barry's gap-sum is itself an arithmetic series over the missing integers in each gap, and so it rests implicitly on the arithmetic mean, specifically, by Proposition~5 of \cite{BarryGapSum},
\[
GS_n = \frac{a_{n+1}-a_n-1}{2}(a_n+a_{n+1}),
\]
which is precisely al-K\=ash\={\i}'s Rule~7 applied to the integers missing between $a_n$ and $a_{n+1}$. \\

In Section~\ref{means-gap-sum}, we extend Barry's gap-sum framework by introducing geometric and harmonic analogues of the gap-sum, completing the Pythagorean triad of means. We show that all three gap-sums admit a unified interpretation within the quasi-arithmetic mean framework of \cite{Bullen2003, NielsenInductiveMean}. In Section~\ref{seq-gap-sum}, we extend al-K\=ash\={\i}'s Rule~7 to ordered sequences of positive integers with arbitrary gaps, deriving a correction formula that accounts for missing terms. The key idea is that the sum of a sparse sequence equals the sum over the full range of consecutive integers, minus the gap-sums of the missing portions. As a concrete application, we derive a product identity involving the Fuss-Catalan numbers. In Section~\ref{sec:applications}, we apply the harmonic gap-sum framework to the Fibonacci and Lucas sequences, establishing exponential convergence of the corresponding reciprocal series and deriving explicit two-term asymptotic expansions with closed-form coefficients.

\section{Gap-Sums via Pythagorean Means}\label{means-gap-sum}

Al-K\=ash\={\i}'s seventh rule computes the sum of consecutive integers present in a sequence via the arithmetic mean of the endpoints. A complementary perspective, introduced by Barry \cite{BarryGapSum}, focuses on the sum of the integers that are not present between consecutive terms of a given sequence. Specifically, Barry defined the \emph{gap-sum} $GS_n$ as
\[
GS_n = \sum_{i=u_n+1}^{u_{n+1}-1} i,
\]
which we interpret as the \emph{arithmetic gap-sum} $GS_n^{(A)}$ to emphasize its connection to the arithmetic mean. Barry proved that
\[
GS_n^{(A)} = \frac{u_{n+1}-u_n-1}{2}(u_n+u_{n+1}).
\]
We now extend this framework to the geometric and harmonic means by considering logarithmic and reciprocal transformations.

\begin{definition}
	For a strictly increasing sequence of positive integers $(u_n)_{n\geq 0}$, we define:
	\begin{enumerate}[(i)]
		\item The \emph{geometric gap-sum}:
		\[
		GS_n^{(G)} = \sum_{i=u_n+1}^{u_{n+1}-1} \log(i) 
		= \log\left(\prod_{i=u_n+1}^{u_{n+1}-1} i\right).
		\]
		\item The \emph{harmonic gap-sum}:
		\[
		GS_n^{(H)} = \sum_{i=u_n+1}^{u_{n+1}-1} \frac{1}{i}.
		\]
	\end{enumerate}
\end{definition}

To unify the three gap-sums, we recall the notion of quasi-arithmetic means.

\begin{definition}[{\cite[Ch.~IV, Def.~1]{Bullen2003}}]
		\label{def:quasi-arithmetic}
		Let $f\colon \mathbb{R}^+ \to \mathbb{R}$ be a strictly monotonic 
		differentiable function. The \emph{quasi-arithmetic mean} of elements 
		$a_1, \ldots, a_n$ with respect to $f$ is defined by
		\[
		M_f(a_1, \ldots, a_n) = f^{-1}\!\left(\frac{1}{n}\sum_{i=1}^n f(a_i)\right).
		\]
		The three classical Pythagorean means correspond to $f(x) = x$ for arithmetic, 
		$f(x) = \log(x)$ for geometric, and $f(x) = 1/x$ for harmonic.
	\end{definition}
    
These definitions yield the following observation.
    
	\begin{observation}
		\label{obs:three-gap-sums}
		Let $(u_n)_{n\geq 0}$ be a strictly increasing sequence of positive integers,
		and let $s_n = u_{n+1} - u_n - 1$ denote the size of the $n$-th gap. 
		The three Pythagorean mean-based gap-sums are:
		\begin{enumerate}[(i)]
			\item Arithmetic: 
			$GS_n^{(A)} = \dfrac{s_n}{2}(u_n+u_{n+1})$
			\item Geometric:
			$GS_n^{(G)} = \log\!\left(\dfrac{(u_{n+1}-1)!}{u_n!}\right)$
			\item Harmonic:
			$GS_n^{(H)} = H_{u_{n+1}-1} - H_{u_n}$, where $H_k = \sum\limits_{i=1}^{k} 
			\frac{1}{i}$ is the $k$-th harmonic number.
		\end{enumerate}
		Moreover, all three satisfy the unified formula
		\[
		GS_n^{(f)} = s_n \cdot f\!\left(M_f(\mathrm{gap})\right),
		\]
        or equivalently, in words,
        \[
        GS_n^{(\mathrm{mean})} = (\text{gap size}) \times 
        (\text{the transformed quasi-arithmetic mean}),
        \]
		where $M_f(\mathrm{gap})$ is the quasi-arithmetic mean of the gap elements
		$\{u_n+1, \ldots, u_{n+1}-1\}$ with respect to $f$.
	\end{observation}
	
	\begin{proof}
		We mentioned earlier the arithmetic case which follows from the standard formula for the sum of an arithmetic progression.
		
		For the geometric case, by definition and standard properties of logarithms, we have
		\[
		GS_n^{(G)} = \sum_{i=u_n+1}^{u_{n+1}-1} \log(i) 
		= \log \prod_{i=u_n+1}^{u_{n+1}-1} i 
		= \log\!\left(\frac{(u_{n+1}-1)!}{u_n!}\right).
		\]
		For the unified formula, since $f(x) = \log(x)$ gives 
		$M_{\log}(\mathrm{gap}) = \exp\!\left(\frac{1}{s_n} 
		\sum\limits_{i=u_n+1}^{u_{n+1}-1} \log(i)\right)$, we have
		\[
		s_n \cdot \log\!\left(M_{\log}(\mathrm{gap})\right) 
		= s_n \cdot \frac{1}{s_n}\sum\limits_{i=u_n+1}^{u_{n+1}-1} \log(i) 
		= GS_n^{(G)}.
		\]
		
		For the harmonic case, by definition and telescoping, we have
		\[
		GS_n^{(H)} = \sum_{i=u_n+1}^{u_{n+1}-1} \frac{1}{i} 
		= \sum_{i=1}^{u_{n+1}-1} \frac{1}{i} - \sum_{i=1}^{u_n} \frac{1}{i} 
		= H_{u_{n+1}-1} - H_{u_n}.
		\]
		For the unified formula, since $f(x) = 1/x$ gives 
		$M_{1/x}(\mathrm{gap}) = \left(\frac{1}{s_n} 
		\sum\limits_{i=u_n+1}^{u_{n+1}-1} \frac{1}{i}\right)^{-1}$, we have
		\[
		s_n \cdot f\!\left(M_{1/x}(\mathrm{gap})\right) 
		= s_n \cdot \frac{1}{s_n}\sum_{i=u_n+1}^{u_{n+1}-1} \frac{1}{i} 
		= GS_n^{(H)}.
		\]
	\end{proof}

    The observation above extends naturally to arbitrary strictly monotonic functions.
    
	\begin{remark}
		\label{rem:f-gap-sum}
		More generally, for any strictly monotonic differentiable function 
		$f\colon \mathbb{R}^+ \to \mathbb{R}$, one can define the $f$-gap-sum as
		\[
		GS_n^{(f)} = \sum_{i=u_n+1}^{u_{n+1}-1} f(i),
		\]
		which always satisfies $GS_n^{(f)} = s_n \cdot f(M_f(\mathrm{gap}))$ 
		by Definition~\ref{def:quasi-arithmetic}. The three Pythagorean gap-sums 
		of Observation~\ref{obs:three-gap-sums} are thus special cases of this 
		general framework.
	\end{remark}

\section{Sparse Sequence Summation}\label{seq-gap-sum}
	
	Having established the gap-sum framework, we now return to the global summation problem. Al-K\=ash\={\i}'s seventh rule provides a closed-form expression for the 
	sum of consecutive integers via the arithmetic mean of the endpoints. In many situations, one encounters ordered integer sequences in which some 
	intermediate values are missing. As we shall see, the correction terms that account for these missing values are precisely the gap-sums introduced in the preceding 
	section. Moreover, this principle extends naturally beyond the arithmetic mean to any strictly monotonic function, yielding a unified sparse summation formula across all three Pythagorean means.
	
	\begin{theorem}\label{thm:sparse-general}
		Let $u_1, u_2, \ldots, u_n$ be a strictly increasing sequence of positive integers 
		with gaps $g_1, \ldots, g_t$, where $g_j = u_{i_j+1} - u_{i_j} > 1$ for 
		$1 \leq i_j \leq n-1$ and $1 \leq j \leq t$. Let $f\colon \mathbb{R}^+ \to 
		\mathbb{R}$ be a strictly monotonic differentiable function. Then
		\[
		\sum_{i=1}^{n} f(u_i) = \sum_{i=u_1}^{u_n} f(i) 
		- \sum_{j=1}^{t} GS_{i_j}^{(f)},
		\]
		where $GS_{i_j}^{(f)} = \sum\limits_{i=u_{i_j}+1}^{u_{i_j+1}-1} f(i)$ is the 
		$f$-gap-sum of the $j$-th gap as defined in Remark~\ref{rem:f-gap-sum}.
	\end{theorem}
	
	\begin{proof}
		The sum of $f(i)$ over all integers from $u_1$ to $u_n$ can be decomposed as
		\[
		\sum_{i=u_1}^{u_n} f(i) = \sum_{i=1}^{n} f(u_i) + \sum_{j=1}^{t} 
		\sum_{i=u_{i_j}+1}^{u_{i_j+1}-1} f(i),
		\]
		since the integers from $u_1$ to $u_n$ consist exactly of the sequence terms 
		$u_1, u_2, \ldots, u_n$ together with the integers lying strictly inside each 
		gap. The inner sum on the right is precisely the $f$-gap-sum $GS_{i_j}^{(f)}$ 
		of the $j$-th gap. Rearranging gives
		\[
		\sum_{i=1}^{n} f(u_i) = \sum_{i=u_1}^{u_n} f(i) - \sum_{j=1}^{t} GS_{i_j}^{(f)}.
		\qedhere
		\]
	\end{proof}
	
	\begin{remark}
		The theorem expresses the sum as
		\[
		\sum_{i=1}^{n} f(u_i) 
		= \bigl(\text{full range sum of } f\bigr)
		- \sum_{j=1}^{t} \bigl(f\text{-gap-sum of gap } j\bigr),
		\]
		that is, the sum over the sparse sequence equals the sum over the full 
		consecutive range, corrected by subtracting the contributions of the missing 
		integers in each gap.
	\end{remark}
	
	The three Pythagorean means correspond to the choices $f(x) = x$, 
	$f(x) = \log(x)$, and $f(x) = 1/x$ respectively. Substituting each into 
	Theorem~\ref{thm:sparse-general} and applying the closed forms of 
	Observation~\ref{obs:three-gap-sums} yields the following explicit formulas.
	
	\begin{corollary}\label{cor:pythagorean-sparse}
		Under the same hypotheses as Theorem~\ref{thm:sparse-general}, the three 
		Pythagorean cases give:
		\begin{enumerate}[(i)]
			\item Arithmetic:
			\[
			\sum_{i=1}^{n} u_i = (u_n-u_1+1)\left(\frac{u_1+u_n}{2}\right)
			- \sum_{j=1}^{t} (g_j-1)\left(u_{i_j}+\frac{g_j}{2}\right).
			\]
			\item Geometric:
			\[
			\sum_{i=1}^{n} \log(u_i) = \log\left(\frac{u_n!}{(u_1-1)!}\right)
			- \sum_{j=1}^{t} \log\left(\frac{(u_{i_j+1}-1)!}{u_{i_j}!}\right).
			\]
			\item Harmonic:
			\[
			\sum_{i=1}^{n} \frac{1}{u_i} = H_{u_n} - H_{u_1-1}
			- \sum_{j=1}^{t} \left(H_{u_{i_j+1}-1} - H_{u_{i_j}}\right).
			\]
		\end{enumerate}
	\end{corollary}
	
	\begin{proof}
		Each case follows by substituting the corresponding function into 
		Theorem~\ref{thm:sparse-general} and applying the closed forms established 
		in Observation~\ref{obs:three-gap-sums}.
		
		\noindent(i) For $f(x) = x$, the full range sum is $\sum\limits_{i=u_1}^{u_n} i = 
		(u_n - u_1 + 1)\left(\frac{u_1+u_n}{2}\right)$ by al-K\=ash\={\i}'s Rule~7, 
		and $GS_{i_j}^{(A)} =\frac{g_j-1}{2} (u_{i_j} + u_{i_j+1})= (g_j-1)\left(u_{i_j} + \frac{g_j}{2}\right),$ by 
		Observation~\ref{obs:three-gap-sums}(i).
		
		\noindent(ii) For $f(x) = \log(x)$, the full range sum is 
		$\sum\limits_{i=u_1}^{u_n} \log(i) = \log\left(\frac{u_n!}{(u_1-1)!}\right)$, 
		and $GS_{i_j}^{(G)} = \log\left(\frac{(u_{i_j+1}-1)!}{u_{i_j}!}\right),$ 
		by Observation~\ref{obs:three-gap-sums}(ii).
		
		\noindent(iii) For $f(x) = 1/x$, the full range sum is 
		$\sum_{i=u_1}^{u_n} \frac{1}{i} = H_{u_n} - H_{u_1-1}$, 
		and $GS_{i_j}^{(H)} = H_{u_{i_j+1}-1} - H_{u_{i_j}},$ 
		by Observation~\ref{obs:three-gap-sums}(iii).
	\end{proof}

    In the following example, we verify case~(i) of Corollary~\ref{cor:pythagorean-sparse} for arithmetic progressions.
    
	\begin{example}
		Suppose 
		\[
		X = u_1 + (u_1 + r) + (u_1 + 2r) + \cdots + \bigl(u_1 + (n-1)r\bigr).
		\]
		With $\min = u_1$ and $\max = u_1 + (n-1)r$, applying 
		Corollary~\ref{cor:pythagorean-sparse}(i) gives
		\[
		X = \bigl((n-1)r + 1\bigr)
		\left(u_1 + \frac{(n-1)r}{2}\right)
		- (r-1)\sum_{j=1}^{n-1} u_j - (r-1)(n-1)\frac{r}{2}.
		\]
		Since $\sum_{j=1}^{n-1} u_j = X - \bigl(u_1 + (n-1)r\bigr)$, substituting 
		and solving for $X$ gives
		\[
		rX = nru_1 + (n-1)r\left(\frac{(n-1)r}{2} + \frac{r-1}{2} + \frac{1}{2}\right),
		\]
		and thus
		\[
		X = nu_1 + \frac{n(n-1)r}{2}.
		\]
		This agrees with al-K\=ash\={\i}'s formula
		\[
		X = \frac{n}{2}\bigl(2u_1 + (n-1)r\bigr) = nu_1 + \frac{n(n-1)r}{2},
		\]
		confirming Corollary~\ref{cor:pythagorean-sparse}(i) for arithmetic progressions.
	\end{example}

		

As a concrete application of the geometric case, we derive the following product identity.

\begin{corollary}
Let $k \geq 1$ be a positive integer. Then
\[
\prod_{i=0}^{n}(ki+1) 
= \frac{(kn+1)!}{\displaystyle\prod_{j=0}^{n-1} k!\cdot FC(k,j)},
\]
where $FC(k,j) = \frac{1}{kj+1}\binom{(k+1)j}{j}$ is the 
$(j,k)$ Fuss--Catalan number.
\end{corollary}


\begin{proof}
	We apply Corollary~\ref{cor:pythagorean-sparse}(ii) to the sequence $a_n = kn+1$. 
	By Proposition~17 of \cite{BarryGapSum}, the gap-product of this sequence is
	\[
	GP_n = k! \cdot FC(k,n),
	\]
	where $FC(k,n) = \frac{1}{kn+1}\binom{(k+1)n}{n}$ is the $(n,k)$ 
	Fuss-Catalan number. Using $GS_n^{(G)} = \log(GP_n)$ by 
	Observation~\ref{obs:three-gap-sums}(ii) and applying 
	Corollary~\ref{cor:pythagorean-sparse}(ii), we have
	\[
	\sum_{i=0}^{n} \log(ki+1) = \log\!\left((kn+1)!\right) - 
	\sum_{j=0}^{n-1} \log\!\left(k! \cdot FC(k,j)\right).
	\]
	Now, by exponentiating both sides of the above equation, we get
	\[
	\prod_{i=0}^{n}(ki+1) = \frac{(kn+1)!}{\displaystyle\prod_{j=0}^{n-1} 
		k! \cdot FC(k,j)}.
	\]
	Finally, substituting $FC(k,j) = \frac{1}{kj+1}\binom{(k+1)j}{j}$ and simplifying the above equation yields the stated identity.
\end{proof}

\section{Applications to Fibonacci and Lucas Sequences} \label{sec:applications}
The Fibonacci numbers, $F_n \;(n\geq 1)$, introduced to Western mathematics by Leonardo of Pisa (known as Fibonacci) in his \textit{Liber Abaci} (1202). \footnote{The sequence was known much earlier in Indian mathematics. Ping\=ala (c.\ 2nd century BCE) is considered the first known authority on these numbers \cite{Singh1985}, 
see also \cite{Shahriari2021}. Shahriari \cite{Shahriari2021} refers to them as the Ping\=ala--Fibonacci numbers.} And the Lucas numbers, $L_n \; (n\geq 1)$, studied by the French mathematician \'{E}douard Lucas (1842--1891), are among the most celebrated sequences in mathematics. Both satisfy the same recurrence $u_n = u_{n-1} + u_{n-2}$ but with different initial conditions, and both have been extensively studied for their arithmetic and analytic properties \cite{Koshy2018}. The sparse summation framework developed in Section~\ref{seq-gap-sum} applies naturally to integer sequences defined by linear recurrences. In this section, we focus on the harmonic case and apply Corollary~\ref{cor:pythagorean-sparse}(iii) to these two classical sequences. In both cases, the harmonic gap-sum decomposition reveals the precise exponential rate at which the corresponding reciprocal series converges and yields asymptotic expansions with explicit closed-form coefficients. Throughout the paper, we use $F_n \; (n\geq 1)$ and $L_n \; (n\geq 1)$ for the nth Fibonacci and Lucas numbers, respectively.

\subsection{Exponential Convergence of the Reciprocal Fibonacci Series} \label{sec:fibon}

We begin by mentioning Binet's formula and its asymptotic consequences,
which underlie all estimates in this section.
Throughout, let
\[
\alpha = \frac{1+\sqrt{5}}{2}, \qquad \beta = \frac{1-\sqrt{5}}{2}
\]
denote the positive and negative roots of $x^2 - x - 1 = 0$, respectively,
so that $\alpha + \beta = 1$, $\alpha\beta = -1$, and $\alpha - \beta = \sqrt{5}$.
Note that $\alpha \approx 1.6180$ is the golden ratio and $\beta \approx -0.6180$,
with $|\beta| < 1 < \alpha$.

		
		

\begin{observation}
		\label{obs:fibonacci-growth}
		For all sufficiently large $n$, $1/F_n = O(\alpha^{-n})$. 
		More precisely, $1/F_n \leq 2\sqrt{5}\,\alpha^{-n}$.
	\end{observation}
	
	\begin{proof}
		By Binet's formula~\cite[Theorem~5.5, p.~90]{Koshy2018},
		\[
		F_n = \frac{\alpha^n - \beta^n}{\sqrt{5}} 
		= \frac{\alpha^n}{\sqrt{5}}\!\left[1 - \left(\frac{\beta}{\alpha}\right)^{\!n}\right],
		\]
		where $|\beta/\alpha| = \alpha^{-2} < 1$, so $(\beta/\alpha)^n \to 0$ as 
		$n \to \infty$~\cite[Section~5.5, p.~101]{Koshy2018}. For all sufficiently 
		large $n$, we have $|(\beta/\alpha)^n| \leq 1/2$, which gives 
		$F_n \geq \alpha^n/(2\sqrt{5})$, and hence $1/F_n \leq 2\sqrt{5}\,\alpha^{-n}$.
	\end{proof}

\medskip

We now apply the harmonic gap-sum decomposition for the Fibonacci sequence to obtain asymptotic information about the rate at which partial sums of $\sum\limits_{i=1}^{\infty} 1/F_i$ converge to the reciprocal Fibonacci constant $\psi = \sum\limits_{i=1}^{\infty} 1/F_i \approx 3.3599$, which is known to be irrational~\cite{AndreJeannin1989}.

\begin{lemma}[Asymptotic harmonic gap-sum  for Fibonacci numbers]\label{lem:gapsum-asymp}
	For the Fibonacci sequence,
	\[
	GS_n^{(H)} \;=\; \ln(\alpha) \;+\; O\!\left(\alpha^{-n}\right),
	\qquad n \to \infty.
	\]
	In particular, $\lim_{n\to\infty} GS_n^{(H)} = \ln(\alpha) \approx 0.4812$.
\end{lemma}

\begin{proof}
	By definition, we have
	\[
	GS_n^{(H)} = \sum_{k=F_n+1}^{F_{n+1}-1} \frac{1}{k}.
	\]
	Since $f(x) = 1/x$ is strictly decreasing for $x>0$, we have 
	\[
	(b-a)f(b) \le \int_a^b f(x)\,dx \le (b-a)f(a).
	\]
	
	Then, for every positive integer $k$,
	\[
	\int_k^{k+1} \frac{dx}{x} \;\leq\; \frac{1}{k}
	\qquad\text{and}\qquad
	\frac{1}{k} \;\leq\; \int_{k-1}^{k} \frac{dx}{x},
	\]
	because on $[k, k+1]$ the function $1/x$ lies below its left-endpoint value
	$1/k$, and on $[k-1,k]$ it lies above its right-endpoint value $1/k$.
	Summing the left inequality from $k = F_n+1$ to $k = F_{n+1}-1$, we have
	\[
	GS_n^{(H)}
	= \sum_{k=F_n+1}^{F_{n+1}-1} \frac{1}{k}
	\;\geq\; \sum_{k=F_n+1}^{F_{n+1}-1} \int_k^{k+1} \frac{dx}{x}
	= \int_{F_n+1}^{F_{n+1}} \frac{dx}{x}
	= \ln\!\left(\frac{F_{n+1}}{F_n+1}\right).
	\]
	Summing the right inequality from $k = F_n+1$ to $k = F_{n+1}-1$, we have
	\[
	GS_n^{(H)}
	= \sum_{k=F_n+1}^{F_{n+1}-1} \frac{1}{k}
	\;\leq\; \sum_{k=F_n+1}^{F_{n+1}-1} \int_{k-1}^{k} \frac{dx}{x}
	= \int_{F_n}^{F_{n+1}-1} \frac{dx}{x}
	= \ln\!\left(\frac{F_{n+1}-1}{F_n}\right).
	\]
	Thus,
	\[
	\ln\!\left(\frac{F_{n+1}}{F_n+1}\right)
	\;\leq\; GS_n^{(H)} \;\leq\;
	\ln\!\left(\frac{F_{n+1}-1}{F_n}\right).
	\]
	
	Rewriting each bound by separating out $F_{n+1}/F_n$, we have
	\begin{align*}
		\ln\!\left(\frac{F_{n+1}}{F_n+1}\right)
		&= \ln\!\left(\frac{F_{n+1}}{F_n}\right)
		+ \ln\!\left(\frac{F_n}{F_n+1}\right)
		= \ln\!\left(\frac{F_{n+1}}{F_n}\right)
		- \ln\!\left(1 + \frac{1}{F_n}\right), \\[6pt]
		\ln\!\left(\frac{F_{n+1}-1}{F_n}\right)
		&= \ln\!\left(\frac{F_{n+1}}{F_n}\right)
		+ \ln\!\left(\frac{F_{n+1}-1}{F_{n+1}}\right)
		= \ln\!\left(\frac{F_{n+1}}{F_n}\right)
		+ \ln\!\left(1 - \frac{1}{F_{n+1}}\right).
	\end{align*}
    By \cite[Corollary~8.6, p.~152]{Koshy2018},
    $F_{n+1}/F_n \to \alpha$. And since $\ln(1+x) = x + O(x^2)$ as $x \to 0$, it follows that
	\[
	\ln\!\left(1+\frac{1}{F_n}\right)
	= O(\alpha^{-n}),
	\qquad
	\ln\!\left(1-\frac{1}{F_{n+1}}\right)
	= O(\alpha^{-n}),
	\]
	because $1/F_n = O(\alpha^{-n})$ and $1/F_{n+1} = O(\alpha^{-n})$.
	Therefore, both bounds are equal $\ln(\alpha) + O(\alpha^{-n})$. Hence, we conclude that $GS_n^{(H)} = \ln(\alpha) + O(\alpha^{-n})$.
\end{proof}

We now use Binet's formula to obtain a more precise asymptotic expansion of the partial sums.

\begin{proposition}[Two-term asymptotic expansion for the reciprocal Fibonacci tail]\label{prop:two-term}
	Let $\psi = \sum\limits_{i=1}^{\infty} 1/F_i$ be the reciprocal Fibonacci constant.
	Then
	\[
	\psi - \sum_{i=1}^{n} \frac{1}{F_i}
	=
	\sqrt{5}\,\alpha \cdot \alpha^{-n}
	+
	\frac{\sqrt{5}}{2\alpha^2} \cdot (-1)^{n+1} \cdot \alpha^{-3n}
	+ O\!\left(\alpha^{-5n}\right),
	\]
	where $\sqrt{5}\,\alpha = (5+\sqrt{5})/2 = \alpha^2 + 1 \approx 3.6180$
	and $\sqrt{5}/(2\alpha^2) \approx 0.4271$.
\end{proposition}

\begin{proof}
	We have the tail as $\psi - \sum\limits_{i=1}^{n} 1/F_i = \sum\limits_{k=n+1}^{\infty} 1/F_k$.
	By Binet’s formula,
	\[
	\frac{1}{F_k}
	= \frac{\sqrt{5}}{\alpha^k - \beta^k}
	= \frac{\sqrt{5}}{\alpha^k}\cdot\frac{1}{1-(\beta/\alpha)^k}
	= \sqrt{5}\,\alpha^{-k}\sum_{m=0}^{\infty}\left(\frac{\beta}{\alpha}\right)^{\!mk},
	\]
	where $(\beta/\alpha)^k = (-1)^k\alpha^{-2k}$ and the geometric series converges
	since $\alpha^{-2} < 1$. Therefore, $\left(\frac{\beta}{\alpha}\right)^{\!mk} = (-1)^{mk}\alpha^{-2mk}.$ By expanding the first three terms $m = 0, 1, 2$, we have explicitly
	\[
	\frac{1}{F_k}
	= \sqrt{5}\,\alpha^{-k}
	+ \sqrt{5}(-1)^k\alpha^{-3k}
	+ \sqrt{5}\,\alpha^{-5k}
	+ \cdots
	\]

    Summing over $k \geq n+1$ and collecting by powers, we have 
	\begin{align*}
		\sum_{k=n+1}^{\infty} \frac{1}{F_k}
		&= \sqrt{5}\sum_{k=n+1}^{\infty}\alpha^{-k}
		+ \sqrt{5}\sum_{k=n+1}^{\infty}(-1)^k\alpha^{-3k}
		+ O\!\left(\alpha^{-5n}\right).
	\end{align*}
	The first geometric series gives
	$\sqrt{5}\cdot\alpha^{-(n+1)}/(1-\alpha^{-1})$.
	From $\alpha^2 = \alpha + 1$, we get $\alpha^{-1} = \alpha - 1$, 
	and then $1 - \alpha^{-1} = 2 - \alpha = \alpha^{-2}$. 
	Hence $\sqrt{5}\cdot\alpha^{-(n+1)}/(1-\alpha^{-1}) = \sqrt{5}\,\alpha\cdot\alpha^{-n}$.\\
	
	The second series converges to
	$\sqrt{5}\cdot(-\alpha^{-3})^{n+1}/(1+\alpha^{-3})$.
	Now, we simplify this quantity using the following facts  \[
	\alpha^{-1} = \alpha - 1
	\;\Rightarrow\;
	\alpha^{-3} = (\alpha - 1)^3 = 2\alpha - 3,
	\]
	
and hence
	
	\[
	1 + \alpha^{-3} = 2\alpha - 2 = 2(\alpha - 1) = \frac{2}{\alpha}.
	\]
	Using these values, we obtain $(\sqrt{5}/(2\alpha^2)) \alpha^{-3n}$ with alternating sign $(-1)^{n+1}$ for the convergence value of the second series.
\end{proof}

An immediate consequence of Proposition~\ref{prop:two-term} is the following exponential convergence result.

\begin{corollary}[Exponential convergence of the reciprocal Fibonacci series]\label{cor:exp-conv}
	The reciprocal Fibonacci series converges exponentially:
	\[
	\psi - \sum_{i=1}^{n} \frac{1}{F_i} = O(\alpha^{-n}).
	\]
	The ratio of consecutive tails satisfies
	\[
	\frac{\psi - \sum\limits_{i=1}^{n+1} 1/F_i}{\psi - \sum\limits_{i=1}^{n} 1/F_i}
	\longrightarrow \frac{1}{\alpha} \approx 0.618 \qquad (n\to\infty).
	\]
\end{corollary}


We exploit the harmonic gap-sum decomposition to derive the following asymptotic formula for $H_{F_n}$.

	
	
	
	
\ \\
The harmonic gap-sum decomposition yields the following asymptotic for $H_{F_n}$.

\begin{corollary}
	\label{cor:HFn-asymp}
	\[
	H_{F_n} \sim n\ln(\alpha).
	\]
	More precisely, $H_{F_n} = n\ln(\alpha) + O(1)$.
\end{corollary}

\begin{proof}
	By Corollary~\ref{cor:pythagorean-sparse} applied to the Fibonacci sequence,
	\[
	\sum_{i=1}^{n}\frac{1}{F_i}
	= H_{F_n} - \sum_{j=2}^{n-1}GS_j^{(H)}.
	\]
	By Lemma~\ref{lem:gapsum-asymp}, each gap-sum satisfies
	$GS_j^{(H)} = \ln(\alpha) + O(\alpha^{-j})$,
	and since $\sum\limits_{j\ge2}\alpha^{-j}$ converges, we have
	\[
	\sum_{j=2}^{n-1}GS_j^{(H)}
	= (n-2)\ln(\alpha) + O(1).
	\]
	Substituting and rearranging gives
	\[
	H_{F_n} = \sum_{i=1}^{n}\frac{1}{F_i} + (n-2)\ln(\alpha) + O(1).
	\]
	Since the reciprocal Fibonacci series converges to 
	$\psi = \sum\limits_{i=1}^{\infty}1/F_i$, Corollary~\ref{cor:exp-conv} 
	gives $\sum\limits_{i=1}^{n}1/F_i = \psi + O(\alpha^{-n})$, and since 
	$O(\alpha^{-n})\subset O(1)$, we obtain
	\[
	H_{F_n} = (n-2)\ln(\alpha) + \psi + O(1) = n\ln(\alpha) + O(1),
	\]
	where the finite constants $-2\ln(\alpha) + \psi$ are absorbed into 
	$O(1)$. In particular, $H_{F_n} \sim n\ln(\alpha)$.
\end{proof}

\begin{remark}
	The same asymptotic can be recovered directly from the classical 
	expansion of harmonic numbers~\cite[Theorem~422]{HardyWright2008},
	\[
	H_m = \ln(m) + \gamma + o(1) \qquad (m\to\infty),
	\]
	combined with $F_n \sim \alpha^n/\sqrt{5}$, which yields the sharper form
	\[
	H_{F_n} = n\ln(\alpha) - \tfrac{1}{2}\ln(5) + \gamma + o(1),
	\]
    where $\gamma \approx 0.5772$ 
is Euler's constant \cite{Euler1744}, 	explicitly identifying the constant. Our approach via the harmonic 
	gap-sum decomposition recovers the same leading asymptotic 
	$H_{F_n} \sim n\ln(\alpha)$ by a different route, arising naturally 
	from the sparse summation framework developed in this paper.
\end{remark}

\begin{remark}[Relation to prior work on reciprocal Fibonacci sums]
	The literature on reciprocal Fibonacci sums has primarily focused on
	the arithmetic structure of the tail.
	For example, Ohtsuka and Nakamura~\cite{OhtsukaNavkamura2009}
	determined
	$\left\lfloor\left(\sum\limits_{k=n}^{\infty}\frac1{F_k}\right)^{-1}\right\rfloor,
	$
	and subsequent work by Lee--Park~\cite{LeePark2020} and
	Marques--Trojov\v{s}k\'{y}~\cite{MarquesTrojovsky2022}
	studied analogous problems for $\sum 1/F_k^2$ and related sequences. Proposition~\ref{prop:two-term} takes a different approach. Using Binet's formula together with the harmonic gap-sum framework developed in this paper, we obtain an explicit
	two-term asymptotic expansion for the partial sums
	$\sum\limits_{i=1}^n 1/F_i$ with closed-form coefficients.
\end{remark}



\subsection{Exponential Convergence of the Reciprocal Lucas Series}
						
	In this section, we study the harmonic gap-sum for the Lucas sequence,
providing a companion to Section~\ref{sec:fibon} that reveals both the structural parallel and the precise numerical distinction between the two 
sequences. The Lucas numbers $(L_n)_{n \geq 1}$ are defined by the 
recurrence $L_n = L_{n-1} + L_{n-2}$ for $n \geq 3$, with initial 
conditions $L_1 = 1$ and $L_2 = 3$ \cite[p.~10]{Koshy2018}, yielding 
the sequence $1, 3, 4, 7, 11, 18, \ldots$. The reciprocal Lucas series 
converges to the reciprocal Lucas constant $\psi_L = \sum\limits_{i=1}^{\infty} 
1/L_i \approx 1.9629$, which is known to be irrational~\cite{AndreJeannin1989}. 
The following observation is the analogue of Observation~\ref{obs:fibonacci-growth} that will be useful in this section.
						
        
        

\begin{observation}
\label{obs:lucas-growth}
For all $n \geq 1$, $1/L_n = O(\alpha^{-n})$.
More precisely, $1/L_n \leq \alpha \cdot \alpha^{-n}$.
\end{observation}

\begin{proof}
By Binet's formula for Lucas numbers~\cite[Theorem~5.7, p.~93]{Koshy2018},
\[
L_n = \alpha^n + \beta^n 
= \alpha^n\!\left[1 + \left(\frac{\beta}{\alpha}\right)^{\!n}\right],
\]
where $L_{n+1}/L_n \to \alpha$ as $n \to \infty$~\cite[Corollary~8.8, p.~153]{Koshy2018}.
Since $|(\beta/\alpha)^n| = \alpha^{-2n} \leq \alpha^{-2} < 1/2$ for all $n \geq 1$,
and $1 - \alpha^{-2} = \alpha^{-1}$, we have
\[
L_n
= \alpha^n\!\left[1+\left(\frac{\beta}{\alpha}\right)^{\!n}\right]
\geq
\alpha^n\!\left(1-\left|\left(\frac{\beta}{\alpha}\right)^{\!n}\right|\right)
\geq
\alpha^n(1-\alpha^{-2})
=
\alpha^{n-1}.
\]
Hence $1/L_n \leq \alpha^{1-n} = \alpha \cdot \alpha^{-n}$.
\end{proof}

Now, we state the asymptotic harmonic gap-sum for Lucas numbers which is similar to Lemma \ref{lem:gapsum-asymp}.

\begin{lemma}[Asymptotic harmonic gap-sum for Lucas numbers]%
    \label{lem:lucas-gap}
    For the Lucas sequence,
    \[
    \mathrm{GS}^{(H)}_n \;=\; \ln(\alpha) + O\!\left(\alpha^{-n}\right),
    \qquad n \to \infty.
    \]
    In particular, $\lim_{n\to\infty}\mathrm{GS}^{(H)}_n = \ln(\alpha)
    \approx 0.4812$.
\end{lemma}

\begin{proof}
    By definition, we have $\mathrm{GS}^{(H)}_n = \sum\limits_{k=L_n+1}^{L_{n+1}-1} 1/k.$
    By replacing $F_n$ by $L_n$ in the integral sandwich of Lemma \ref{lem:gapsum-asymp}, it gives
    \[
    \ln\!\frac{L_{n+1}}{L_n+1}
    \;\leq\; \mathrm{GS}^{(H)}_n
    \;\leq\; \ln\!\frac{L_{n+1}-1}{L_n}.
    \]
    By separating out $L_{n+1}/L_n$ in each bound, we have
    \begin{align*}
        \ln\left(\frac{L_{n+1}}{L_n+1}\right)
        &= \ln\left(\frac{L_{n+1}}{L_n}\right) - \ln\!\Bigl(1+\tfrac{1}{L_n}\Bigr), \\
        \ln\left(\frac{L_{n+1}-1}{L_n} \right)
        &= \ln\left(\frac{L_{n+1}}{L_n}\right) + \ln\!\Bigl(1-\tfrac{1}{L_{n+1}}\Bigr).
    \end{align*}
    Since $L_{n+1}/L_n \to \alpha$ (by ~\cite[Corollary~8.8, p.~153]{Koshy2018}) and $1/L_n = O(\alpha^{-n})$
    which are in Observation \ref{obs:lucas-growth}, both logarithmic correction terms are $O(\alpha^{-n})$,
    and the squeeze theorem gives
    $\mathrm{GS}^{(H)}_n = \ln(\alpha) + O(\alpha^{-n})$.
\end{proof}

We now sharpen this estimate for the Lucas tail.

\begin{proposition}[Two-term asymptotic expansion for the reciprocal Lucas tail]
    \label{prop:lucas-tail}
    Let $\psi_L = \sum\limits_{i=1}^{\infty} 1/L_i$ be the reciprocal Lucas
    constant.  Then
    \[
    \psi_L - \sum_{i=1}^{n} \frac{1}{L_i}= 
    \alpha \cdot \alpha^{-n}
	+
	\frac{1}{2\alpha^2} \cdot (-1)^{n+1} \cdot \alpha^{-3n}
	+ O\!\left(\alpha^{-5n}\right),
    \]
    where $\alpha \approx 1.6180$ and $\frac{1}{2\alpha^2} \approx 0.1910$.
\end{proposition}

\begin{proof}
Write the tail as $\psi_L - \sum\limits_{i=1}^{n} 1/L_i = \sum\limits_{k=n+1}^{\infty} 1/L_k$.
	By Binet’s formula for Lucas numbers, we have
\[
\frac{1}{L_k}
= \frac{1}{\alpha^k + \beta^k}
= \alpha^{-k}\cdot\frac{1}{1+(-(\beta/\alpha)^k)}
= \alpha^{-k}\sum_{m=0}^{\infty}\left(-\left(\frac{\beta}{\alpha}\right)^{k}\right)^m,
\]
 where the series converges since $|\frac{\beta}{\alpha}| = \alpha^{-2} < 1$. Using $(\frac{\beta}{\alpha})^k = (-1)^k\alpha^{-2k}$, we expand the sum
 \[
\frac{1}{L_k}
= \alpha^{-k}
+ (-1)^{k+1}\alpha^{-3k}
+ \alpha^{-5k}
+ \cdots
\]
Summing over $k \geq n+1$ and collecting by powers, we have
\begin{align*}
\sum_{k=n+1}^{\infty} \frac{1}{L_k}
&= \sum_{k=n+1}^{\infty}\alpha^{-k}
+ \sum_{k=n+1}^{\infty}(-1)^{k+1}\alpha^{-3k}
+ O\!\left(\alpha^{-5n}\right).
\end{align*}
The first geometric series on the right-hand side of the equality converges to
$\alpha\alpha^{1-n}$.
Now, since $1-\alpha^{-1} = 2- \alpha$, $\alpha^2(2-\alpha)=1$ and $2-\alpha = \alpha^{-2}$, we deduce that  
\[
\sum_{k=n+1}^{\infty}\alpha^{-k} = \frac{\alpha^{-(n+1)}}{1-\alpha^{-1}} = \frac{\alpha^{-(n+1)}}{\alpha^-2} = \alpha^{1-n}.
\]
The second sum $\sum\limits_{k=n+1}^{\infty}(-1)^{k+1}\alpha^{-3k} = -\sum\limits_{k=n+1}^{\infty}(-1)^{k}\alpha^{-3k}$ gives
\[
\sum_{k=n+1}^{\infty}(-1)^{k}\alpha^{-3k} = \sum_{k=n+1}^{\infty}(-\alpha^{-3})^k = \frac{(-\alpha^{-3})^{n+1}}{1-\alpha^{-3}}.
\]
From $\alpha^3 = \alpha \cdot \alpha^2 = \alpha(\alpha+1) = \alpha^2 + \alpha = 2\alpha+1$, we have  $1+\alpha^{-3} = \frac{\alpha^3+1}{\alpha^3} = \frac{2(\alpha+1)}{\alpha^3} = \frac{2(\alpha^2)}{\alpha^3} = \frac{2}{\alpha}$. Hence,
\[
\frac{(-\alpha^{-3})^{n+1}}{1-\alpha^{-3}} = \frac{(-1)^{n+1}\alpha^{-3(n+1)}}{2/\alpha}= \frac{(-1)^{n+1}}{2\alpha^2}\cdot \alpha^{-3n}.
\]

Therefore, we obtain
\[
\sum_{k=n+1}^{\infty} \frac{1}{L_k} = \alpha^{1-n} + \frac{(-1)^{n+1}}{2\alpha^2}\cdot \alpha^{-3n} +O\!\left(\alpha^{-5n}\right).
\]

\end{proof}

\begin{corollary}[Exponential convergence of the reciprocal Lucas series]
    \label{cor:lucas-exp}
    The reciprocal Lucas series converges exponentially to $\psi_L$. More specifically, 
    \[
    \psi_L - \sum_{i=1}^{n}\frac{1}{L_i} = O(\alpha^{-n}).
    \]
    The ratio of consecutive tails satisfies
    \[
    \frac{\displaystyle\psi_L - \sum_{i=1}^{n+1} 1/L_i}
    {\displaystyle\psi_L - \sum_{i=1}^{n} 1/L_i}
    \;\longrightarrow\; \frac{1}{\alpha} \approx 0.618 \qquad n\to\infty.
	\]
    
\end{corollary}

Similar to the Corollary~\ref{cor:HFn-asymp}, as a consequence of the harmonic gap-sum decomposition, we obtain the following asymptotic 
formula for $H_{L_n}$, which is new to the best of our knowledge. Although the Lucas numbers satisfy the same recurrence as the Fibonacci numbers, this result does not follow immediately from Corollary~\ref{cor:HFn-asymp}. The key difference lies in the initial conditions. Since $F_1 = F_2 = 1$, the Fibonacci sequence has no first gap and the decomposition starts at $j = 2$, whereas 
$L_1 = 1$ and $L_2 = 3$ produce a non-trivial first gap, so the Lucas decomposition starts at $j = 1$. This shifts the absorbed $O(1)$ constant from $-2\ln(\alpha) + \psi$ to $-\ln(\alpha) + \psi_L$, and the proof requires Lemma~\ref{lem:lucas-gap} independently.

\begin{corollary}
	\label{cor:HLn-asymp}
    We have
	\[
	H_{L_n} \sim n\ln(\alpha).
	\]
	More precisely, $H_{L_n} = n\ln(\alpha) + O(1)$.
\end{corollary}

\begin{proof}
	By Corollary~3.3 applied to the Lucas sequence,
	\[
	\sum_{i=1}^{n}\frac{1}{L_i}
	= H_{L_n} - H_{L_1-1} - \sum_{j=1}^{n-1}\mathrm{GS}^{(H)}_j.
	\]
	Since $L_1 = 1$, we have $H_{L_1-1} = H_0 = 0$, which simplifies to
	\[
	\sum_{i=1}^{n}\frac{1}{L_i}
	= H_{L_n} - \sum_{j=1}^{n-1}\mathrm{GS}^{(H)}_j.
	\]
	By Lemma~\ref{lem:lucas-gap}, each gap-sum satisfies
	$\mathrm{GS}^{(H)}_j = \ln(\alpha) + O(\alpha^{-j})$,
	and since $\sum\limits_{j\ge1}\alpha^{-j}$ converges, we have
	\[
	\sum_{j=1}^{n-1}\mathrm{GS}^{(H)}_j = (n-1)\ln(\alpha) + O(1).
	\]
	Substituting and rearranging gives
	\[
	H_{L_n} = \sum_{i=1}^{n}\frac{1}{L_i} + (n-1)\ln(\alpha) + O(1).
	\]
	Since the reciprocal Lucas series converges to 
	$\psi_L = \sum\limits_{i=1}^{\infty}1/L_i$, Corollary~\ref{cor:lucas-exp} 
	gives $\sum\limits_{i=1}^{n}1/L_i = \psi_L + O(\alpha^{-n})$, and since 
	$O(\alpha^{-n})\subset O(1)$, we obtain
	\[
	H_{L_n} = (n-1)\ln(\alpha) + \psi_L + O(1) = n\ln(\alpha) + O(1),
	\]
	where the finite constants $-\ln(\alpha) + \psi_L$ are absorbed into 
	$O(1)$. In particular, $H_{L_n} \sim n\ln(\alpha)$.
\end{proof}

Alternatively, substituting $L_n \sim \alpha^n$ directly into the classical expansion $H_m = \ln(m) + \gamma + o(1)$ \cite[Theorem~422]{HardyWright2008} recovers $H_{L_n} \sim n\ln(\alpha)$, though without identifying the constant $\psi_L - \ln(\alpha)$.


    \subsection{Identities for Euler's Constant via Harmonic Gap-Sums}
    
    The proofs of Corollaries~\ref{cor:HFn-asymp} and 
    \ref{cor:HLn-asymp} contain additional information beyond the 
    asymptotic formulas $H_{F_n} \sim n\ln(\alpha)$ and 
    $H_{L_n} \sim n\ln(\alpha)$. By comparing the two asymptotic 
    expansions of $H_{F_n}$, one obtained via the harmonic gap-sum decomposition and one from the classical Hardy--Wright expansion, we extract an explicit formula for Euler's 
    constant $\gamma$.
    
    \begin{proposition}\label{prop:euler-fibonacci}
    	We have
    	\[
    	\gamma = \psi + \tfrac{1}{2}\ln(5) - 2\ln(\alpha) 
    	+ \sum_{j=2}^{\infty}\bigl(\mathrm{GS}_j^{(H)} - \ln(\alpha)\bigr),
    	\]
    	equivalently, since 
    	$\mathrm{GS}_j^{(H)} = H_{F_{j+1}-1} - H_{F_j}$ by 
    	Observation~\ref{obs:three-gap-sums}, this can be written as
    	\[
    	\gamma = \psi + \tfrac{1}{2}\ln(5) - 2\ln(\alpha) 
    	+ \sum_{j=2}^{\infty}\Bigl(H_{F_{j+1}-1} - H_{F_j} 
    	- \ln(\alpha)\Bigr).
    	\]
    \end{proposition}
    
    \begin{proof}
    	By Lemma~\ref{lem:gapsum-asymp}, each harmonic gap-sum 
    	satisfies $\mathrm{GS}_j^{(H)} = \ln(\alpha) + O(\alpha^{-j})$, 
    	so $\mathrm{GS}_j^{(H)} - \ln(\alpha) = O(\alpha^{-j})$. Since 
    	$\sum\limits_{j=2}^{\infty}\alpha^{-j}$ converges, the series 
    	$\sum\limits_{j=2}^{\infty}(\mathrm{GS}_j^{(H)} - \ln(\alpha))$ 
    	converges absolutely. From the proof of Corollary~\ref{cor:HFn-asymp}, we have
    	\[
    	H_{F_n} = \sum_{i=1}^{n}\frac{1}{F_i} 
    	+ \sum_{j=2}^{n-1}\mathrm{GS}_j^{(H)}.
    	\]
    	We rewrite the gap-sum as
    	\[
    	\sum_{j=2}^{n-1}\mathrm{GS}_j^{(H)} 
    	= (n-2)\ln(\alpha) 
    	+ \sum_{j=2}^{n-1}\bigl(\mathrm{GS}_j^{(H)} - \ln(\alpha)\bigr),
    	\]
    	and substitute to obtain
    	\[
    	H_{F_n} = \sum_{i=1}^{n}\frac{1}{F_i} + (n-2)\ln(\alpha) 
    	+ \sum_{j=2}^{n-1}\bigl(\mathrm{GS}_j^{(H)} - \ln(\alpha)\bigr).
    	\]
    	As $n \to \infty$, we have $\sum\limits_{i=1}^{n} 1/F_i \to \psi$ 
    	by Corollary~\ref{cor:exp-conv}, with 
    	$\sum\limits_{i=1}^{n} 1/F_i = \psi + O(\alpha^{-n})$. 
    	Moreover, since $\sum\limits_{j=2}^{\infty}(\mathrm{GS}_j^{(H)} - \ln(\alpha))$ 
    	converges absolutely, as established above, the tail satisfies
    	\[
    	\sum_{j=2}^{n-1}\bigl(\mathrm{GS}_j^{(H)} - \ln(\alpha)\bigr) 
    	= \sum_{j=2}^{\infty}\bigl(\mathrm{GS}_j^{(H)} - \ln(\alpha)\bigr) 
    	- \sum_{j=n}^{\infty}\bigl(\mathrm{GS}_j^{(H)} - \ln(\alpha)\bigr),
    	\]
    	where the remainder satisfies 
    	$\sum\limits_{j=n}^{\infty}(\mathrm{GS}_j^{(H)} - \ln(\alpha)) 
    	= \sum\limits_{j=n}^{\infty}O(\alpha^{-j}) = O(\alpha^{-n}) \subset o(1)$ 
    	as $n \to \infty$. Therefore,
    	\[
    	\sum_{j=2}^{n-1}\bigl(\mathrm{GS}_j^{(H)} - \ln(\alpha)\bigr) 
    	= \sum_{j=2}^{\infty}\bigl(\mathrm{GS}_j^{(H)} - \ln(\alpha)\bigr) 
    	+ o(1).
    	\]
    	Hence,
    	\[
    	H_{F_n} = n\ln(\alpha) + \psi - 2\ln(\alpha) 
    	+ \sum_{j=2}^{\infty}\bigl(\mathrm{GS}_j^{(H)} - \ln(\alpha)\bigr) 
    	+ o(1).
    	\]
    	On the other hand, substituting $F_n \sim \alpha^n/\sqrt{5}$ 
    	into the classical expansion 
    	$H_m = \ln(m) + \gamma + o(1)$~\cite[Theorem~422]{HardyWright2008} 
    	gives
    	\[
    	H_{F_n} = n\ln(\alpha) - \tfrac{1}{2}\ln(5) + \gamma + o(1).
    	\]
    	Since both expansions hold as $n \to \infty$, comparing the 
    	constant terms yields
    	\[
    	\gamma = \psi + \tfrac{1}{2}\ln(5) - 2\ln(\alpha) 
    	+ \sum_{j=2}^{\infty}\bigl(\mathrm{GS}_j^{(H)} - \ln(\alpha)\bigr),
    	\]
    	as claimed.
    \end{proof}
    
    Applying the same argument to the Lucas sequence via 
    Corollary~\ref{cor:HLn-asymp} and substituting 
    $L_n \sim \alpha^n$ into the classical expansion yields the 
    following Lucas analogue.
    
    \begin{proposition}\label{prop:euler-lucas}
    	We have
    	\[
    	\gamma = \psi_L - \ln(\alpha) 
    	+ \sum_{j=1}^{\infty}\bigl(\mathrm{GS}_j^{(H)} - \ln(\alpha)\bigr),
    	\]
    	equivalently,
    	\[
    	\gamma = \psi_L - \ln(\alpha) 
    	+ \sum_{j=1}^{\infty}\Bigl(H_{L_{j+1}-1} - H_{L_j} 
    	- \ln(\alpha)\Bigr).
    	\]
    \end{proposition}
    
    \begin{proof}
    	From the proof of Corollary~\ref{cor:HLn-asymp}, we have
    	\[
    	H_{L_n} = \sum_{i=1}^{n}\frac{1}{L_i} 
    	+ \sum_{j=1}^{n-1}\mathrm{GS}_j^{(H)}.
    	\]
    	Rewriting the gap-sum as
    	\[
    	\sum_{j=1}^{n-1}\mathrm{GS}_j^{(H)} 
    	= (n-1)\ln(\alpha) 
    	+ \sum_{j=1}^{n-1}\bigl(\mathrm{GS}_j^{(H)} - \ln(\alpha)\bigr),
    	\]
    	and substituting gives
    	\[
    	H_{L_n} = \sum_{i=1}^{n}\frac{1}{L_i} + (n-1)\ln(\alpha) 
    	+ \sum_{j=1}^{n-1}\bigl(\mathrm{GS}_j^{(H)} - \ln(\alpha)\bigr).
    	\]
    	As $n \to \infty$, since $\sum\limits_{i=1}^{n} 1/L_i = \psi_L 
    	+ O(\alpha^{-n})$ by Corollary~\ref{cor:lucas-exp}, and 
    	$\sum\limits_{j=1}^{\infty}(\mathrm{GS}_j^{(H)} - \ln(\alpha))$ 
    	converges by Lemma~\ref{lem:lucas-gap}, we obtain
    	\[
    	H_{L_n} = n\ln(\alpha) + \psi_L - \ln(\alpha) 
    	+ \sum_{j=1}^{\infty}\bigl(\mathrm{GS}_j^{(H)} - \ln(\alpha)\bigr) 
    	+ o(1).
    	\]
    	On the other hand, substituting $L_n \sim \alpha^n$ into the 
    	classical expansion 
    	$H_m = \ln(m) + \gamma + o(1)$~\cite[Theorem~422]{HardyWright2008} 
    	gives
    	\[
    	H_{L_n} = n\ln(\alpha) + \gamma + o(1).
    	\]
    	Comparing the constant terms yields
    	\[
    	\gamma = \psi_L - \ln(\alpha) 
    	+ \sum_{j=1}^{\infty}\bigl(\mathrm{GS}_j^{(H)} - \ln(\alpha)\bigr),
    	\]
    	as claimed.
    \end{proof}
    
    Equating the two expressions for $\gamma$ in 
    Propositions~\ref{prop:euler-fibonacci} 
    and~\ref{prop:euler-lucas} yields the following relation 
    between the reciprocal Fibonacci and Lucas constants.
    
    \begin{corollary}\label{cor:psi-relation}
    	We have
    	\[
    	\psi - \psi_L = \ln(\alpha) - \tfrac{1}{2}\ln(5) 
    	+ \sum_{j=1}^{\infty}\Bigl(H_{L_{j+1}-1} - H_{L_j} 
    	- \ln(\alpha)\Bigr)
    	- \sum_{j=2}^{\infty}\Bigl(H_{F_{j+1}-1} - H_{F_j} 
    	- \ln(\alpha)\Bigr).
    	\]
    \end{corollary}

\section{Conclusion}
	
	In this paper, we have extended the classical summation framework of al-K\=ash\={\i} through the language of quasi-arithmetic means. In Section~2, we extended Barry's arithmetic gap-sum to include geometric and harmonic versions. The key observation is that all three fit naturally into the theory of quasi-arithmetic means. Each gap-sum turns out to be the gap size times the right kind of mean of the missing integers. In Section~3, we proved a general sparse summation theorem that expresses the sum of a strictly monotonic function over a sparse integer sequence as the full range sum minus the gap-sums of the missing portions. When we specialize to the arithmetic case, we recover al-K\=ash\={\i}'s Rule~7 exactly. As a concrete application of the geometric case, we derived a product identity involving the Fuss-Catalan numbers. In Section~4, we applied the harmonic gap-sum framework to two classical sequences. For both the Fibonacci and Lucas numbers, we established that the harmonic gap-sum converges to $\ln(\alpha)$ exponentially. For the Fibonacci sequence, we derived an explicit two-term asymptotic expansion for the tail of the reciprocal Fibonacci series with closed-form coefficients, and established the exponential convergence rate and the asymptotic formula $H_{F_n} \sim n\ln(\alpha)$. These results complement the existing literature \cite{LeePark2020, MarquesTrojovsky2022, OhtsukaNavkamura2009}, which has focused on related but distinct problems, such as floor functions of tail reciprocals and reciprocal sums of squares, rather than the two-term asymptotic expansion of the partial sums $\sum_{i=1}^n 1/F_i$ themselves. To the best of our knowledge, the corresponding results for the Lucas sequence, the two-term asymptotic expansion, the exponential convergence rate, and the asymptotic formula $H_{L_n} \sim n\ln(\alpha)$, are new. Furthermore, by comparing the harmonic gap-sum expansions with the classical Hardy--Wright expansion of harmonic numbers, we derived exact series identities expressing Euler's constant $\gamma$ in terms of harmonic numbers at Fibonacci and Lucas indices (Propositions~4.13 and~4.14), and obtained a new identity relating the reciprocal Fibonacci constant $\psi$ and the reciprocal Lucas constant $\psi_L$ Corollary~4.15).
    
	\vspace{.1in}
	We conjecture that the same methods will work for other sequences, including Pell and Tribonacci numbers. It would also be interesting to study the probabilistic side of the quasi-arithmetic gap-sum framework. Also, the connection between the Fuss-Catalan product identity and combinatorial structures invites further exploration. 

\section*{Acknowledgments}
	
	The first and second authors gratefully acknowledge support from the Global Scholar Award, University of Evansville Center for Innovation and Change.


\begin{thebibliography}{99}

\bibitem{AndreJeannin1989}
R.~Andr\'{e}-Jeannin,
\newblock Irrationalit\'{e} de la somme des inverses de certaines suites
r\'{e}currentes,
\newblock \emph{C.~R. Acad. Sci. Paris S\'{e}r.~I Math.}, \textbf{308}(19)
(1989), 539--541.

\bibitem{nuh}
Aydin, N., Azarian, M. K., Khormali, O. and Mtmet, G., \textit{On the History of the Square and Multiply Algorithm}, arXiv:2606.00958.

\bibitem{Miftah1}
Aydin, N. and Hammoudi, L.
\textit{Al-K\={a}sh\={i}'s Mift\={a}\d{h} al-\d{H}is\={a}b, Volume~I: Arithmetic}.
Birkh\"{a}user, 2019.

\bibitem{Miftah2}
Aydin, N., Hammoudi, L., and Bakbouk, G.
\textit{Al-K\={a}sh\={i}'s Mift\={a}\d{h} al-\d{H}is\={a}b, Volume~II: Geometry}.
Birkh\"{a}user, 2020.



\bibitem{alkashi_miftah}
N.~Aydin, L.~Hammoudi, and G.~Bakbouk,
\newblock \emph{Al-K\={a}sh\={i}'s Mift\={a}h al-His\={a}b, Volume~III:
Algebra},
\newblock Springer, 2022.


\bibitem{Az1}
Azarian, M.~K.
\textit{An Overview of Mathematical Contributions of Ghiy\={a}th al-D\={i}n Jamsh\={i}d Al-K\={a}sh\={i} [K\={a}sh\={a}n\={i}]}.
Mathematics Interdisciplinary Research, \textbf{4} (2019), 11--19.

\bibitem{Az2}
Azarian, M.~K.
\textit{A Study of Ris\={a}la al-Watar wa'l Jaib (``The Treatise on the Chord and Sine''): Revisited}.
Forum Geometricorum, \textbf{18} (2018), 219--222.
Mathematical Reviews MR3819227; Zentralblatt MATH Zbl~1395.51026.

\bibitem{Az3}
Azarian, M.~K.
\textit{A Study of Ris\={a}la al-Watar wa'l Jaib (``The Treatise on the Chord and Sine'')}.
Forum Geometricorum, \textbf{15} (2015), 229--242.
Mathematical Reviews MR3418854; Zentralblatt MATH Zbl~1328.01015.

\bibitem{Az4}
Azarian, M.~K.
\textit{Al-Ris\={a}la al-Mu\d{h}\={\i}\d{t}\={\i}yya: A Summary}.
Missouri Journal of Mathematical Sciences, 22(2):64--85, 2010.
Mathematical Reviews MR2675403; Zentralblatt MATH Zbl~1204.01009.  Historia Mathematica, \textbf{38}(3) (2011), p. 435.

\bibitem{Az6}
Azarian, M.~K. \textit{Al-K\={a}sh\={i}'s  Fundamental Theorem}. International Journal of Pure and Applied Mathematics, 14(4):499--509, 2004. Mathematical Reviews, MR2005b:01021 (01A30), February 2005, p. 919.  Zentralblatt MATH, Zbl 1059.01005.  Historia Mathematica, \textbf{32}(4) (2005), p. 500.


\bibitem{Az}
Azarian, M.~K.
\textit{Meft\={a}\d{h} al-\d{h}es\={a}b [Mift\={a}\d{h} al-\d{h}is\={a}b]: A Summary}.
Missouri Journal of Mathematical Sciences, 12(2):75--95, 2000.
Mathematical Reviews MR1764526; Zentralblatt MATH Zbl~1036.01002.
Historia Mathematica, \textbf{31}(2) (2004), p. 233.






\bibitem{BarryGapSum}
P.~Barry,
\newblock On the gap-sum and gap-product sequences of integer sequences,
\newblock arXiv:2104.05593, 2021.

\bibitem{Bullen2003}
P.~S.~Bullen,
\newblock \emph{Handbook of Means and Their Inequalities},
\newblock Mathematics and Its Applications, Vol.~560, Springer, 2003.
\newblock Revised edition of the 1988 original.

\bibitem{Euler1744}
L.~Euler,
\newblock Variae observationes circa series infinitas,
\newblock \emph{Commentarii Acad. Sci. Petropolitanae}, \textbf{9} (1744),
160--188.
\newblock (Presented 1737.)

\bibitem{HardyWright2008}
G.~H.~Hardy and E.~M.~Wright,
\newblock \emph{An Introduction to the Theory of Numbers},
\newblock 6th ed., Oxford University Press, Oxford, 2008.

\bibitem{Koshy2018}
T.~Koshy,
\newblock \emph{Fibonacci and Lucas Numbers with Applications, Volume~One},
\newblock 2nd ed., John Wiley \&\ Sons, Hoboken, New Jersey, 2018.

\bibitem{LeePark2020}
H.-H.~Lee and J.-D.~Park,
\newblock Asymptotic behavior of reciprocal sum of two products of Fibonacci
numbers,
\newblock \emph{J. Inequal. Appl.}, \textbf{2020}, Article 91 (2020).

\bibitem{MarquesTrojovsky2022}
D.~Marques and P.~Trojovsk\'{y},
\newblock The proof of a formula concerning the asymptotic behavior of the
reciprocal sum of the square of multiple-angle Fibonacci numbers,
\newblock \emph{J. Inequal. Appl.}, \textbf{2022}, Article 21 (2022).

\bibitem{Nicomachus}
Nicomachus of Gerasa, \textit{Introduction to Arithmetic}, 
translated by M.~L.~D'Ooge, Macmillan, New York, 1926.

\bibitem{NielsenInductiveMean}
F.~Nielsen,
\newblock What is an inductive mean?
\newblock \emph{Notices Amer. Math. Soc.}, \textbf{70}(11) (2023),
1851--1855.

\bibitem{OhtsukaNavkamura2009}
H.~Ohtsuka and S.~Nakamura,
\newblock On the sum of reciprocal {F}ibonacci numbers,
\newblock \emph{Fibonacci Quart.}, \textbf{46/47}(2) (2008/2009), 153--159.


\bibitem{Shahriari2021}
S.~Shahriari, \textit{An Invitation to Combinatorics}, Cambridge Mathematical Textbooks, Cambridge University Press, Cambridge, 2021.

\bibitem{Singh1985}
P.~Singh, The so-called Fibonacci numbers in ancient and medieval India, \textit{Historia Mathematica}, \textbf{12}(3) (1985), 229--244.

\end{thebibliography}
\end{document}